\documentclass[11p]{elsarticle}
\usepackage{amssymb,amsmath,amsthm,latexsym}
\usepackage{amsfonts,amssymb,amsthm}
\usepackage{amsfonts,mathtools}
\usepackage{graphicx,color}
\newtheorem{thm}{Theorem}

\newtheorem{theorem}{Theorem}[section]
\newtheorem{lemma}[theorem]{Lemma}
\newtheorem{claim}[theorem]{Claim}

\theoremstyle{definition}

\journal{the journal.}

\begin{document}
	
		\begin{frontmatter}
	
	\title{\textbf{Decompositions of graphs with degree constraints relative to prescribed subgraphs}}
	
	\author{Peichao Wei\corref{cor1}}
	\ead{wei229090@163.com}
	\address{Institute of Mathematics, School of Mathematical Sciences, Nanjing Normal University, 
		Nanjing, 210023, China}
	
	\author{Muhuo Liu\corref{cor2}}
	\ead{liumuhuo@163.com}
	\address{Department of Mathematics,  South China     Agricultural  University, Guangzhou, 510642, China}
	
	\author{Yang Wu\corref{cor3}}
	\ead{ywu@must.edu.mo}
	\address{School of Computer Science and Engineering,   Macau University of Science and Technology, Macau, China}
	
	\author{Zoran Stani\' c\corref{cor4}}
		\ead{zstanic@matf.bg.ac.rs}
		\address{Faculty of Mathematics, University of Belgrade,
			Studentski trg 16, 11 000 Belgrade, Serbia}
	
		\cortext[cor2]{Corresponding author}

%
%
	
	\begin{abstract}
Given a finite simple undirected graph $G$, let $T_1(G)$ denote the  subset of vertices of $G$ such that every vertex of $T_1(G)$ belongs to at least one subgraph isomorphic to a graph obtained by connecting a single vertex to two vertices of $K_4 - e$. Define  $T_0(G) = V(G) \setminus T_1(G)$, and let $a,b \colon V(G) \longrightarrow \mathbb{Z}_{\ge 0}$ be arbitrary functions. In this paper, we prove that if $d_G(u) \ge a(u) + b(u) + h(u)$, where $h(u) \in \{0,1\}$ for $u \in T_h(G)$, then there exists a partition $(S, T)$ of $V(G)$ such that $d_{S}(u) \ge a(u)$ for every $u \in S$ and $d_{T}(u) \ge b(u)$ for every $u \in T$. This result extends the theorem of Stiebitz~[\textit{J. Graph Theory}, 23  (1996), 321--324]. Moreover, we establish an analogous result in the case where $T_1(G)$ consists of vertices belonging to at least one $K_{2,3}$, thereby extending the findings of Hou et al.~[\textit{Discrete Math.}, 341 (2018), 3288--3295].
\end{abstract}


\end{frontmatter}

\let\thefootnote\relax\footnotetext{\textit{ORCIDs}: \texttt{0009-0002-8309-9984} (P. Wei) \texttt{000-0001-8217-2452} (M. Liu), (Y. Wu), \texttt{0000-0002-4949-4203} (Z. Stani\' c)}

	 \section{Introduction}

	 All graphs considered in this paper are finite simple and undirected. For such a graph $G$, we denote its vertex set and edge set by $V(G)$ and $E(G)$, respectively. For a vertex $u \in V(G)$, let $N_{G}(u)$ and $d_{G}(u)$ denote the neighbourhood and degree of $u$, respectively; thus $N_{G}(u) = \{ v \in V(G) ~:~ uv \in E(G) \}$ and $d_{G}(u) = |N_{G}(u)|$. The minimum degree of $G$ is written as $\delta(G)$. For a vertex subset $X$ and a vertex $v$ of $G$, we define $d_X(u) = |N_G(u) \cap X|$. Moreover, $G[X]$ denotes the subgraph of $G$ induced by $X$. For a fixed graph $H$, $G$ is said to be \textit{$H$-free} if it does not contain $H$ as a subgraph (not necessarily induced). For terminology and notation not explicitly defined in this paper, we refer the reader to \cite{Diestel}.
	 
	 Many fundamental and applied problems in graph theory revolve around the concept of graph partitioning, where the vertex set of a graph is divided into parts satisfying prescribed structural or combinatorial properties; among these, one of the most extensively studied directions involves partitions subject to minimum degree constraints.~To provide a concise overview of results relevant to this study, we first introduce the following notions.

	Let $\mathbb{Z}_{\ge 0}$ denote the set of non-negative integers. For a function $f \colon V(G) \longrightarrow \mathbb{Z}_{\ge 0}$ defined on the vertex set of an arbitrary graph $G$, a subset $X \subseteq V(G)$ is said to be \emph{$f$-feasible} if $d_X(u) \ge f(u)$ for every $u \in X$. Similarly, given two functions $a,b \colon V(G) \longrightarrow \mathbb{Z}_{\ge 0}$, a pair $(A,B)$ of non-empty disjoint subsets of $V(G)$ is called \emph{$(a,b)$-feasible} if $A$ is $a$-feasible and $B$ is $b$-feasible. An $(a,b)$-feasible pair satisfying $A \cup B = V(G)$ is referred to as an \emph{$(a,b)$-feasible partition}.  
	
	Throughout this paper, $K_4-e$ denotes the graph obtained by deleting a single edge from the complete graph~$K_4$, while $K_{2,3}$ denotes the complete bipartite graph with partite sets of cardinalities $2$ and $3$, respectively.

For positive integers $s$ and $t$, let $g(s,t)$ denote the minimum integer such that 
every graph $G$ with minimum vertex degree $\delta(G) \ge g(s,t)$ admits a vertex partition 
$(S,T)$ satisfying $\delta(G[S]) \ge s$ and $\delta(G[T]) \ge t$. 
In 1983, Thomassen~\cite{Thomassen} conjectured that 
\[
g(s,t) = s + t + 1,
\]
and thirteen years later this conjecture was confirmed by Stiebitz~\cite{Sti}.

	 \begin{thm}[\cite{Sti}]\label{thm:A} For a graph $ G $ and $ a,b\colon V(G)\longrightarrow \mathbb{Z}_{\ge 0}$, 
	 	suppose that $ d_G(u) \ge a(u) + b(u) + 1 $ holds for every $ u  \in V (G) $. Then $ G $ has an $ (a,b) $-feasible partition.\end{thm}
	  	 
	 In subsequent work, a series of Stiebitz-type results has appeared. The degree condition in Theorem~\ref{thm:A} can be relaxed when we impose restrictions on the (local) structure of the target graphs. 
	 For instance, Kaneko~\cite{Kaneko} showed that 
	 $
	 g(s,t) \le s + t
	 $
	 holds whenever $G$ is triangle-free. Next, Hou et al.~\cite{Hou} established the following results.

	 \begin{thm}[\cite{Hou}]\label{thm:B} For a $ K_{2,3} $-free graph $G$ and $ a,b\colon V(G)\longrightarrow \mathbb{Z}_{\geq 1}$, if $ d_G(u) \ge  a(u) + b(u) $ holds for every
	 	$ u \in  V(G)$, then $ G $ admits an $ (a, b) $-feasible partition.
	 \end{thm}
	
	\begin{thm}[\cite{Hou}]\label{thm:C} For a graph $ G $  with at least five vertices and  $ a,b\colon V(G)\longrightarrow \mathbb{Z}_{\geq 0}$, suppose that $ G $ has no subgraph
		obtained by connecting a single vertex to exactly two vertices of $ K_{4} -e $  and  $ d_G(u) \ge  a(u) + b(u) $ holds for every
		$ u \in  V(G)$. Then $ G $ admits an $ (a, b) $-feasible partition.
		\end{thm}
	 
These results are followed by additional findings that, while less central to the focus of the present paper, provide useful insights into related problems; in the following, we highlight the following.

In 2000, Diwan~\cite{Diwan} proved that 
$
g(s,t) \le s + t - 1
$ 
holds for graphs with girth at least five whenever $\min\{s,t\} \ge 2$. Gerber and Kobler~\cite{Kobler} generalized Diwan's result to a specific function-type setting, not elaborated here. 
Another noteworthy contribution is due to Bazgan, Tuza and Vanderpooten~\cite{Vander}, who presented polynomial-time algorithms to find $(a,b)$-feasible partitions satisfying the results of Stiebitz, Kaneko and Diwan, respectively. Ma et al.~\cite{Ma} further extended Diwan's result by showing that 
$
g(s,t) \le s + t - 1
$ 
holds for $C_4$-free graphs whenever $\min\{s,t\} \ge 2$. 
Recently, Zeng and Zu~\cite{Zu} proved that a graph in which quadrangles do not share edges with triangles and other quadrangles admits an  $ (a,b) $-feasible partition whenever $ d_{G}(u)\geqslant a(u)+b(u)-1 $ holds for every  $ u\in V(G) $.

We proceed with a slightly extended  approach introduced by Furuya et al.~\cite{Fu}. Let $ J(G) $  be the set of pairs  $ (u_{1}u_{2}\ldots u_{s},v_{1}v_{2}\ldots v_{t}) $  of vertex sequences of $ G $  such that (a) $ \{u_{i}:1\leq i\leq s\}\neq\{v_{i}:1\leq i\leq t\} $, (b) $s\in\{3,4\} $  and  $t\in\{3,4\} $, (c) $ (u_{1},u_{2})=(v_{1},v_{2}) $ and (d) each sequence induces a  cycle. Let  $S_{1}(G)=\{u_{1},u_{2}~:~(u_{1}u_{2}\ldots u_{m},v_{1}v_{2}\ldots v_{n})\in J(G)\} $  and  $ S_{0}(G)=V(G)\setminus S_{1}(G) $.

\begin{thm}[\cite{Fu}]\label{thm:D}
		For a graph $ G $  and  $ a,b\colon V(G)\longrightarrow \mathbb{Z}_{\geq 0}$, suppose that $ d_G(u)\ge a(u)+b(u)-1+2h(u)$ and  $\min \left\lbrace a(u),b(u)\right\rbrace\ge 2(1-h(u))$ for $u\in S_{h(u)}(G)$, with $ h(u)\in \left\lbrace 0,1\right\rbrace  $. Then $G$ admits an $ (a,b) $-feasible partition.
\end{thm}

By analyzing the preceding theorem, one may expect that a graph $G$ admits an $(a,b)$-feasible partition 
whenever the vertices associated with a given subgraph have comparatively large degrees. 
This observation motivates the main result of the present work.

The following notation will be used throughout the remainder of the paper without further explicit reference.
 If $H$ is a prescribed subgraph of $G$,  let $T_1(G)$ denote the subset of $V(G)$ such that every vertex of $T_1(G)$ belongs to at least one subgraph isomorphic to $H$. Set $T_0(G) = V(G) \setminus T_1(G)$.

\begin{theorem}\label{13t}
	For a graph $ G $  with at least five vertices and $ a,b\colon V(G)\longrightarrow \mathbb{Z}_{\ge 0}$, suppose that  $ d_G(u)\ge a(u)+b(u)+h(u)$ holds  for $u\in T_{h(u)}(G)$, with $ h(u)\in \left\lbrace 0,1\right\rbrace$. Then $G$ admits an  $(a,b)$-feasible partition whenever one of the following holds:
	\begin{itemize} \item[(i)] $\min \left\lbrace a(u),b(u)\right\rbrace\ge 1-h(u)$ for $u\in T_{h(u)}(G)$ and $H$ is  obtained   by connecting a single vertex to exactly two vertices of $ K_{4} -e $;
		\item[(ii)] $\min \left\lbrace a(u),b(u)\right\rbrace\ge 2-h(u)$ for $u\in T_{h(u)}(G)$  and  $H\cong K_{2,3}$.
		\end{itemize}
\end{theorem}	

Evidently, Theorem~\ref{13t} extends the result of Theorem~\ref{thm:A}, as it requires a weaker initial assumption, $
d_G(u) \ge a(u) + b(u) + h(u)$
instead of
$
d_G(u) \ge a(u) + b(u) + 1,
$
while introducing an additional condition specified in~(i) or~(ii). Moreover, part~(i) (resp.~part~(ii)) permits the existence of subgraphs isomorphic to a single vertex joined by two edges to $K_4 - e$ ($K_{2,3}$), thereby extending the results of Theorem~\ref{thm:C} (Theorem~\ref{thm:B}). This highlights how Theorem~\ref{13t} compares to the previous results by  relaxing the degree requirement and allowing specific subgraph configurations that were not covered in earlier theorems. 

The initial framework, together with relevant results, notions and notation, is presented in Section~\ref{sec:prep}. The proof of Theorem~\ref{13t} is provided in Section~\ref{sec:proof}.

	\section{Preparation}\label{sec:prep}

Everything stated in this section will be used in the next one. In particular, the initial assumptions given after  the next lemma remain in force throughout the remainder of the paper. 

\begin{lemma}[\cite{Sti}]\label{21l}
	For a graph $ G $ and  $ a,b\colon V(G)\longrightarrow \mathbb{Z}_{\ge 0}$, $ G $ admits an $ (a,b) $-feasible partition if and only if $ G $ admits an $ (a,b) $-feasible pair.
\end{lemma}

Henceforth, we assume that $ G $ has no $ (a,b) $-feasible partition, and therefore  
$ G $ has no $ (a,b) $-feasible pair, by the previous lemma. Next, we choose $ a $ and $ b $ so that $$\sum_{u\in V(G)}\big(a(u)+b(u)\big)$$ is as large as possible, subject to the non-existence of $(a,b)$-feasible pairs. For the sake of simplicity, for  $ h\in\{0,1\} $, we set $ T_h\coloneqq T_h(G) $. 

We proceed with the following notation and notions.

\begin{itemize}
	\item For $ X \subseteq V(G) $, let $ B_1(X)=\left\lbrace u\in X~:~d_{X}(u)\le a(u)\right\rbrace  $ and $ B_2(X)=\left\lbrace u\in X~:~d_{X}(u)\le b(u)+h(u)-1\right\rbrace  $.
	\item For $ i\in\left\lbrace 1,2\right\rbrace$, a non-empty subset $ X_i $ of $ V(G) $ is called \textit{$i$-good} if  $B_i(X_i)=\emptyset$.
	
	\item   For $ i\in\left\lbrace 1,2\right\rbrace  $, a non-empty subset $ X_i $ of $ V(G) $ is called \textit{$i$-meager} if $ B_i(X')\neq\emptyset $ holds for every non-empty subset $ X' $ of $ X_i$.
	\item A partition $ (X_1,X_2) $ of $ V(G) $ is called \textit{degenerate} if $ X_i $ is  $i$-meager for each $ i\in\{1,2\} $.
	
	\item A non-empty subset $ X $ of $ V(G) $ is called \textit{$ a$-degenerate}  if  for every non-empty subset $ X' $ of $ X $  there exists a vertex $ x\in X' $ such that $ d_{X'}(x)\le a(x) $.
	
	\item A non-empty subset $ X $ of $ V(G) $ is called \textit{$a$-nice}   if $  d_{X}(x)\ge a(x) $ holds for every $ x\in X $.	 \end{itemize}

The subsets that are $b$-degenerate or $b$-nice are defines analogously.  Clearly, if $ X_1 $ is $ 1 $-good, then  $ X_1 $ is $a$-nice. Similarly, if $ X_2 $ is $ 2 $-good, then  $ X_2 $ is $b$-nice.

For two vertices $ u,v\in V(G) $, let $ e(u,v)$ denote the number of edges incident with each. Clearly, $ e(u,v)\in\{0,1\} $. For $X\subseteq V(G)$, $ e(X) $ denotes the number of edges with each end in $ X $. 

For a partition $ (X_1,X_2) $ of $ V(G) $, a  \textit{weight} $\omega(X_1,X_2)$ is defined as
\begin{align}
	\omega(X_1,X_2)=e(X_1)+e(X_2)+\sum_{u\in X_1}b(u)+\sum_{v\in X_2}a(v).
\end{align}
For  $ u\in B_1(X_1) $,  $ v\in B_2(X_2) $, a straightforward algebraic calculation shows that
\begin{align}\label{23a}
	\omega(X_1\setminus\{u\},X_2\cup\{u\})-\omega(X_1,X_2)&=d_{X_2}(u)-d_{X_1}(u)+a(u)-b(u)\nonumber\\&\ge (b(u)+h(u))-a(u)+a(u)-b(u)\\&=h(u),\nonumber
\end{align}
and
\begin{align}\label{24a}
	\omega(X_1\cup\{v\},X_2\setminus\{v\})-\omega(X_1,X_2)&=d_{X_1}(v)-d_{X_2}(v)+b(v)-a(v)\nonumber\\& \ge(a(v)+1)-(b(v)+h(v)-1)+b(v)-a(v)\\&=2-h(v)\nonumber.
\end{align}
The exchange of  $ u $  and  $ v $  yields a new partition:  $ X_{1}'=X_{1}\cup\{v\}\setminus\{u\} $,  $ X_{2}'=X_{2}\cup\{u\}\setminus\{v\} $.  It follows from \eqref{23a} and \eqref{24a} that
\begin{align}\label{25a}\nonumber
	w (X_{1}',X_{2}')-\omega(X_{1},X_{2}) &=\left(d_{X_{2}}(u)-d_{X_{1}}(u)+a(u)-b(u)\right)+\left(d_{X_{1}}(v)-d_{X_{2}}(v)-a(v)+b(v)\right)-2e(u,v)\\&\ge\,2-h(v)+h(u)-2e(u,v).
\end{align}

Note that in the previous computation, the value of $h(u)$ is predetermined and depends solely on whether $u$ belongs to $T_1$; that is, it is independent of whether $u$ lies in $X_1$ or $X_2$. In the next section, we proceed within the same framework assuming a partition $(T_0, T_1)$, which uniquely determines the value of $h(u)$ for every~$u$. The graph \( H \) that determines the preceding partition is either left unspecified or fixed to one of the graphs described in Theorem~\ref{13t}. In the former case, the partition \( (T_0, T_1) \) is not uniquely determined but is merely guaranteed to exist.

	 \section{Proof}\label{sec:proof}
	 
	In light of Theorems \ref{thm:B} and \ref{thm:C},  we may suppose that $T_0\ne \emptyset$ and $T_1\ne \emptyset$. 
	
	In what follows, we prove a sequence of claims, each under the assumptions introduced after Lemma~\ref{21l}. Taken together, these establish the proof of Theorem~\ref{13t}.

The first three claims examine the vertices of $G$.

	\begin{claim}\label{21c}
		For every vertex $ u\in V(G) $, $ d_G(u)=a(u)+b(u)+h(u) $.
	\end{claim}
\begin{proof}
	Suppose that $ d_G(u)\neq a(u)+b(u)+h(u) $ for some $u\in V(G)$. Then, $ d_G(u)\ge a(u)+b(u)+h(u)+1 $. Let $ a'\colon V(G)\longrightarrow\mathbb{Z}_{\ge 0}$ be a function satisfying \begin{align*}
		a'(v)=\begin{cases}
			a(v)+1, &v=u\\a(v), &v\neq u.
		\end{cases}
	\end{align*}
From $ d_G(u)\ge a(u)+b(u)+h(u)+1 =a'(u)+b(u)+h(u)$, we deduce that $ a' $ and $ b $  satisfy   $$\sum_{v\in V(G)}(a'(v)+b(v))>\sum_{v\in V(G)}(a(v)+b(v)).$$ However, this contradicts  the maximality of $ \sum_{v\in V(G)}(a(v)+b(v)) $.
\end{proof}

\begin{claim}\label{222c}
	For every vertex $ u\in V(G) $, $\min\{a(u),b(u)\}\ge1$.
\end{claim}
\begin{proof}
		 If  $\min\{a(u),b(u)\}\ge2-h(u)$, then this claim holds. Otherwise,  $\min\{a(u),b(u)\}\ge1-h(u)$.  By symmetry, suppose that $ a(u)\le0 $. Hence, $ 1-h(u)\le a(u)\le0 $, and so $ h(u)\ge1 $. This, together with  definition of $ h(u) $, forces $ h(u)=1 $ and $ a(u)=0 $.

 Let $ X_1=\{u\} $ and $ X_2=V(G)\setminus\{u\} $.  Then $$ d_{X_2}(v)\ge a(v)+b(v)+h(v)-e(u,v)\ge 1-h(v)+b(v)+h(v)-1=b(v)$$ holds for every vertex $ v\in X_2 $. Therefore, $ (X_1,X_2) $ is  an $ (a,b) $-feasible partition of $ G $,  which contradicts the non-existence of such a partition (assumed in the previous section).
\end{proof}
\begin{claim}\label{22c}
		For every vertex $ u\in V(G) $, $  V(G)\setminus \left\lbrace u \right\rbrace  $ is  $ 2 $-good.
		\end{claim}
	\begin{proof}	For every vertex $ v\in V(G)\setminus\left\lbrace u \right\rbrace  $,   Claim \ref{222c} ensures     \begin{align*}
			d_{V(G)\setminus \left\lbrace u \right\rbrace }(v)\ge (a(v)+b(v)+h(v))-e(v,u)
			\ge b(v)+1+h(v)-1=b(v)+h(v).
		\end{align*}
			Since $ v $ is arbitrary, $ V(G)\setminus \left\lbrace u \right\rbrace $ is $ 2 $-good, as desired.	\end{proof}

We now show the existence of a degenerate partition.

\begin{claim}\label{23c}
		There exists a degenerate partition of $ V(G) $.
	\end{claim} \begin{proof}
	By Claim \ref{22c}, there exists a proper subset of $ V(G) $ that is $ 2 $-good. Let $ X_2 $ be a 2-good subset of $V(G)$ with the minimum cardinality. By the minimality of $ X_2 $, every proper non-empty subset of $ X_2 $ is 2-meager. Let $ X_1=V(G)\setminus X_2 $. Since $ X_2 $ is a proper subset of $ V(G) $, $ X_1\neq\emptyset $. If there exists a $ 1 $-good subset $ X_1' $ of $ X_1 $, then $ (X_1',X_2) $ is an $ (a,b) $-feasible pair of $ G $, which is impossible. Hence, $ X_1 $ is 1-meager.
	
	The minimality of $ X_2 $ yields the existence of a vertex $ u\in X_2 $ with $ d_{X_2}(u)=b(u)+h(u) $. Indeed, if $ d_{X_2}(w)\ge b(w)+h(w)+1 $ holds for every vertex $ w\in X_2 $, then   $ X_2\setminus\{w\} $ would also be a  $ 2 $-good subset of $ V(G) $. 
	
	We set $ X_1''=X_1\cup\left\lbrace u \right\rbrace$ and $ X_2''=X_2\setminus \left\lbrace u \right\rbrace  $. Obviously, $(X_1'', X_2'')$ is a partition of $V(G)$, and to conclude the proof it is sufficient to show that $X_i''$ is $i$-meager for $i\in\{1,2\}$. Since $$ \left| X_2\right|\ge d_{X_2}(u)+1= b(u)+h(u)+1\ge 1-h(u)+h(u)+1=2,$$ $ X_2'' $ is non-empty. Hence, $ X_2'' $ is $ 2 $-meager. On the other hand, $d_{X_1''}(u)=d_{X_1}(u)=a(u)$ by Claim \ref{21c}. Since $ X_1 $ is $ 1 $-meager,  $ X_1'' $ is also $ 1 $-meager, and we are done.
\end{proof}

Let  $ \mathcal{P} $ denote the family of degenerate partitions $ (X_1,X_2) $ of $ V(G) $ satisfying the following two conditions:
\begin{itemize}
	\item $ \omega(X_1,X_2) $ is maximum;
	\item  $ \left| X_1\right|  $ is  minimum,  subject to the previous item.
\end{itemize}
By Claim \ref{23c}, $ \mathcal{P}\neq\emptyset $. In addition,  $ X_1 $  is  $ (a-1) $-degenerate  or  $ X_2 $  is  $ (b-1) $-degenerate, as  $ G $  does not admit a feasible partition.

 \begin{claim}\label{28c}
If $ (X_1,X_2)\in\mathcal{P} $, then 	$\min\{ |X_1|,\left|X_2 \right|\} \ge2 $.
\end{claim}
\begin{proof} Let   $ u\in B_2(X_2) $. By Claim \ref{21c}, we have   $$ \left| X_1 \right|\ge e(u,X_1)\ge d_G(u)-(b(u)+h(u)-1)\ge a(u)+1\ge2.$$ 

To complete the argument it suffices to establish the same inequality for $|X_2|$.  Assume that $|X_2|=1$, i.e. $X_2=\{u\}$.  We begin by deriving a collection of structural constraints. Using these constraints, we then construct one or more $(a,b)$-feasible partitions which contradict the assumption, thereby excluding the case $|X_2|=1$.

If $\min \left\lbrace a(v),b(v)\right\rbrace\ge 2-h(v)$ holds for every $v\in X_1$, then  $$d_{X_1}(v)\ge a(v)+b(v)+h(v)-e(u,v)\ge a(v)+2-h(v)+h(v)-e(u,v)\ge a(v)+1,$$ for every vertex $ v\in X_1 $. However, this contradicts the assumption $(X_1,X_2)\in\mathcal{P}$. Therefore, the assumption (ii) of Theorem~\ref{13t} is not satisfied, meaning the defining graph of the set $T_1$, denoted by $H$, is as in (i): obtained   by connecting a single vertex to exactly two vertices of $ K_{4} -e $. 

   Since $ X_1 $ is $ 1 $-meager, there is a vertex $x\in B_1(X_1)$. Since  $ d_{X_2}(x)\ge b(x)+h(x)\ge 1 $, by Claim \ref{222c},  we have $ux\in E(G)$. If $ (X_1\setminus\{x\}, X_2\cup\{x\}) $ is a degenerate partition, then the inequality \eqref{23a} implies $$ \omega(X_1\setminus\{x\}, X_2\cup\{x\})-\omega(X_1,X_2)\ge h(x)\ge 0,$$ which contradicts  the minimality of $| X_1|$. Thus, the partition is not degenerate, which implies that $X_2\cup\{x\}$ is $2$-good, and so $1=d_{\{u,x\}}(u)\ge b(u)+h(u)$ and
$1=d_{\{u,x\}}(x)\ge b(x)+h(x)$. Hence, $b(u)=b(x)=1$ and $h(u)=h(x)=0$, by Claim \ref{222c}. This, in particular, means that $x,u\in T_0$. In addition, for every  $x\in B_1(X_1)$, from  $ux\in E(G)$  and $d_{\{u,x\}}(x)=b(x)=1$, by employing  Claim \ref{22c} we obtain $d_{X_1}(x)=a(x)$.

From $x\in T_0$ and (again) $ux\in E(G)$, we have $$\left| N_{X_1}(x)\cap B_1(X_1)\right|\le 2,$$
as  $ \left| N_{X_1}(x)\cap B_1(X_1)\right|\ge 3 $ would imply $ x\in T_1(G) $.

 If $ \left| N_{X_1}(x_0)\cap B_1(X_1)\right|=0$ for some $x_0\in B_1(X_1)$, then $(X_1\setminus \{x_0\}, X_2\cup \{x_0\})$ is an $(a,b)$-feasible partition, which is impossible. We select $y_0\in N_{X_1}(x_0)\cap B_1(X_1)$.  By the symmetry of $x_0$ and $y_0$, we have $1\le  \left| N_{X_1}(y_0)\cap B_1(X_1)\right|\le 2$.

We claim that  $|N_{X_1}(y_0)\cap B_1(X_1)|=|N_{X_1}(x_0)\cap B_1(X_1)|=1$. The other possibility for these cardinalities is $2$. Without loss of generality, suppose that   $ N_{X_1}(y_0)\cap B_1(X_1)=\{x_0,w_0\} $.
Since $|V(G)|\ge 5$, we have  $X_1\ne \{w_0,y_0,x_0\}$.
Let $x\in X_1\setminus \{w_0,y_0,x_0\}$. If $x\in B_1(X_1)$, then $N_G(x)\cap \{w_0,y_0,x_0\}=\emptyset$ holds by the structure of $H$, as $x_0\in T_0$. If   $x\not\in B_1(X_1)$, then $|N_G(x)\cap \{w_0,y_0,x_0\}|\le 1$, by the same argument  and $d_{X_1}(x)\ge a(x)+1$. Let $X''_1=X_1\setminus \{w_0,y_0,v_0\}$ and $X''_2=X_2\cup \{w_0,y_0,x_0\}$. Irrespective of whether $x\in B_1(X_1)$, we have $d_{X''_1}(x)\ge a(x)$ and $d_{X''_2}(y)\ge 2>b(y)$ for every $y\in X''_2$. Thus, $(X''_1, X''_2)$ is an $(a, b) $-feasible partition, which is impossible. This confirms our claim.

  On the basis of $ |V(G)|\ge 5 $, let $x\in X_1\setminus\{x_0,y_0\}\neq\emptyset $.
  If $ \left| N(x)\cap \{x_0,y_0\}\right|\le 1$ holds  for every $x\in X_1\setminus\{x_0,y_0\}$, then     $(X_1\setminus\{x_0,y_0\}, X_2\cup\{x_0,y_0\})$ is an $ (a,b) $-feasible partition. Thus, there exists   $x\in X_1\setminus B_1(X_1)$ such that  $ \left| N(x)\cap \{x_0,y_0\}\right|=2$.   Since $u\in T_0$,  $x$ features as the unique vertex of  $ X_1\setminus B_1(X_1) $ with  $ \left| N(x)\cap \{x_0,y_0\}\right|=2$.

  Finally, if $ d_{X_1}(x)\ge a(x)+2 $, then $ (X_1\setminus\{x_0,y_0\}, X_2\cup\{x_0,y_0\}) $ is an $ (a,b) $-feasible partition. Hence, $d_{X_1}(x)=a(x)+1 $. From $ X_1\setminus\{x,y_0,x_0\}\neq\emptyset $, we have $ \left|N(v)\cap\{x,y_0,x_0\} \right|\le 1 $ for every $ v\in X_1\setminus\{x,y_0,x_0\}$, as $u\in T_0$. Hence, $ X_1\setminus\{x,y_0,x_0\} $ is $ a $-nice. It  is easy check  that $(X_1\setminus\{x,y_0,x_0\}, X_2\cup \{x,y_0,x_0\})$ features as an $(a,b)$-feasible partition, which is the final contradiction. \end{proof}

We proceed by proving the existence of a particular good subset.

	\begin{claim}\label{26c}
		Let $ (X_1,X_2)\in\mathcal{P} $ and $ i\in\{1,2\} $. For $ u_i\in B_i(X_i) $, there exists a $ (3-i) $-good subset of $ X_{3-i}\cup\{u_i\} $ containing $ u_i $.
	\end{claim}
\begin{proof}
     It follows from \eqref{23a} and \eqref{24a} that
     \begin{align*}
     	\omega(X_1\setminus\{u_1\},X_2\cup\{u_1\})-\omega(X_1,X_2)\ge h(u) \ge0,
     \end{align*}
 and
 \begin{align*}
 	\omega(X_1\cup\{u_2\},X_2\setminus\{u_2\})-\omega(X_1,X_2)\ge2-h(v)\ge1.
 \end{align*}
This, together with the maximality of $ \omega(X_1,X_2) $ and the minimality of $\left|X_1 \right|$, implies that the partitions $ (X_1\setminus\{u_1\},X_2\cup\{u_1\}) $ and $ (X_1\cup\{u_2\},X_2\setminus\{u_2\}) $ are not degenerate.

 Since $ \left| X_1\right|\ge2 $, we have  $ X_1\setminus\{u_1\}\neq\emptyset $. Observing that  $ X_1 $ is 1-meager, we find that the same holds for $ X_1\setminus\{u_1\} $, and so  $ X_2\cup\{u_1\} $ is not 2-meager. Hence, there exists a  $ 2 $-good subset $ X_2' $ of $ X_2\cup\{u_1\} $. Since $ X_2 $ is 2-meager, we have $ u_1\in X_2' $, and so the desired result holds for $i=1$. The case $ i=2 $ is considered analogously.
\end{proof}

Within the same degenerate partition, we now establish a structural characterisation of the vertices of $B(X_1)$, formulated in relation to the vertices of $B_1(X_2)$.

	 \begin{claim}\label{24c}
	 	For  $ (X_1,X_2)\in\mathcal{P} $, every vertex in $ B_1(X_1) $ is adjacent to all vertices of $ B_2(X_2) $.
	 \end{claim}
	\begin{proof} By Claim \ref{28c}, we have    $\min\{|X_1|,   | X_2|\}\ge 2$. Suppose that there {exists}   $ u_1\in B_1(X_1) $ and $ u_2\in B_2(X_2)$ such that $u_1u_2\not\in E(G)$.  Let $ X_1'=X_1\setminus\{u_1\} $ and $ X_2'=X_2\cup\{u_1\} $. By Claim \ref{26c},  there exists a $ 2 $-good subset $ B $ of $ X_2\cup\{u_{1}\} $ with $ u_{1}\in B $.  Since $ u_1u_2\notin E(G) $ and $ d_{X_2}(u_2)\le b(u_2)+h(u_2)-1 $, we have  $ u_2\notin B $.
 
 Analogously, by considering  $ (X_1\cup\{u_2\},X_2\setminus\{u_2\}) $, we deduce the existence of a $ 1 $-good subset $ A $ of $ X_1\cup\{u_2\} $, with $ u_{2}\in A $ and $u_1\not\in A$.

Thus, $ (A,B) $ features as an $ (a,b) $-feasible pair of $ G $, and we are done.
	\end{proof}

We next examine a perturbation of the initial degenerate partition, achieved through the exchange of two designated vertices.

	\begin{claim}\label{25c}
		Let $ (X_1,X_2)\in\mathcal{P} $ and  $ u_i\in B_i(X_i) $. If either $ u_1,u_2\in T_0$ or $u_1\in T_1$, then $ ((X_1\setminus\left\lbrace u_1\right\rbrace )\cup\left\lbrace u_2\right\rbrace,(X_2\setminus\left\lbrace u_2\right\rbrace )\cup\left\lbrace u_1\right\rbrace )\in\mathcal{P} $. In addition, $u_1\in T_1$ implies  $u_2\in T_1$.
	\end{claim}
	 \begin{proof}
	 	 Let $ X_i'=(X_i\setminus\{u_i\})\cup\{u_{3-i}\} $.  We first prove that $ (X_1',X_2') $ is a degenerate partition. It suffices to show that $ X_i' $ is $i$-meager for each $ i\in\{1,2\} $.

We  show that $ X_2' $ is 2-meager. For the sake of contradiction, we assume the opposite.
 Since $ X_2 $ is 2-meager, there exists a  $ 2$-good subset $ B $ of $ X_2' $ with $ u_{1}\in B $.   It follows from Claim \ref{26c} that there {exists} a $ 1 $-good subset $ A $ of  $ X_1\cup\{u_2\} $ containing $ u_2 $.  If $ u_1\notin A $, then $ (A,B) $ is an $ (a,b) $-feasible pair of $ G $. Otherwise, $$ a(u_1)+1\le d_{X_1\cup\{u_2\}}(u_1)=d_{X_1}(u_1)+e(u_1,u_2)\le a(u_1)+1$$ forces  $ d_{X_1}(u_1)=a(u_1) $. Hence, $ d_{X_2}(u_1)=b(u_1)+h(u_1) $ and $ d_{X_2\setminus\{u_2\}}(u_1)=b(u_1)+h(u_1)-1 $, but this violates the assumption $ u_1\in B $.
	 	
 We now show that $ X_1' $ is 1-meager. Again we assume the opposite. Since $ X_1 $ is 1-meager, there exists a  $ 1$-good subset $ A $ of $ X_1' $ with $ u_{2}\in A $.   It follows from Claim \ref{26c} that there {exists} a $ 2 $-good subset $ B $ of  $ X_2\cup\{u_1\} $ containing $ u_1 $.  The case $ u_2\notin B $ would imply that $ (A,B) $ {is an $ (a,b) $-feasible pair of $ G $}. Therefore, we have $$ b(u_2)+h(u_2)\le d_{X_2\cup\{u_1\}}(u_2)=d_{X_2}(u_2)+e(u_1,u_2)\le b(u_2)+h(u_2)-1+1=b(u_2)+h(u_2),$$ along with $ d_{X_2}(u_2)=b(u_2)+h(u_2)-1 $. Hence, $ d_{X_1}(u_2)=a(u_2)+1 $ and $ d_{X_1'}(u_2)=d_{X_1\setminus\{u_1\}}(u_2)=a(u_2) $, violating $ u_2\in A $. 
 
 In conclusion, $ (X_1',X_2') $ is a degenerate partition of $ V(G) $. Since $u_i \in B_i(X_i)$ and either $u_1, u_2 \in T_0$ or $u_1 \in T_1$, inequality~\eqref{25a} yields
  $$ \omega(X_1',X_2')- \omega(X_1,X_2)\ge2-h(u_2)+h(u_1)-2e(u_1,u_2)=h(u_1)-h(u_2)\ge 0.$$   By the maximality of $ \omega(X_1,X_2) $, we have  $ \omega(X_1,X_2)= \omega(X_1',X_2')  $, where  $u_1\in T_1$ implies $u_2\in T_1$.
	 \end{proof}

\begin{claim}\label{29c}
	Let $ (X_1,X_2)\in\mathcal{P} $. If $u\in B_1(X_1)$ and $v\in B_2(X_2)$, then  $ d_{X_1}(u)=a(u) $ and $ d_{X_2}(v)=b(v)+h(v)-1 $.
\end{claim}
\begin{proof}
 Let $ X_1'=X_1\cup\{v\} $ and $ X_2'=X_2\setminus\{v\} $, where $v\in B_2(X_2)$.   By Claim \ref{28c}, we have  $|X_2|\ge 2$ and $|X_1|\ge 2$. By Claim \ref{26c},  there exists $ A\subseteq X_1 $ such that $ A\cup\{v\} $ is $ 1 $-good,  which implies that $ A $ is $a$-nice.

  We prove that $X_1$ is $a$-nice. Otherwise, there exists a vertex $ x\in X_1\setminus A $ such that $ d_{X_1}(x)\le a(x)-1 $. Let $ X_1''=X_1\setminus\{x\}  $ and $ X_2''=X_2\cup\{x\} $. By employing \eqref{23a}, we find   \begin{align*}
		\omega(X_1'',X_2'')-\omega(X_1,X_2)&=d_{X_2}(x)-d_{X_1}(x)+a(x)-b(x)\nonumber\\&\ge (b(x)+h(x)+1)-(a(x)-1)+a(x)-b(x)\\&=2+h(x) \ge2.
	\end{align*}  Since $X''_1$ is 1-meager, there exists a $ 2 $-good subset $ B\subseteq X_2'' $ such that $ x\in B $. Observing that $ x\notin A $, we deduce that $ (A,B) $ is an $ (a,b) $-feasible pair. Hence, $X_1$ is   $a$-nice, and so $ d_{X_1}(u)=a(u)$.

 Claim \ref{21c}, together with $ d_{X_1}(u)=a(u)$, yields $d_{X_2}(u)=b(u)+h(u)$. Since  $u\in B$ (by Claim \ref{26c}), we deduce that $v\in B$ for every $v\in B_2(X_2)$ (by Claim \ref{24c}), which implies  $ d_{X_2}(v)=b(v)+h(v)-1 $. \end{proof}

Collecting all preceding observations, we arrive at the final step and thus complete the proof.

\medskip\noindent{\textit{Proof of Theorem~\ref{13t}.}} If $B_2(X_2)\subseteq T_1$, then $ (X_1,X_2) $ is an $ (a,b) $-feasible partition by Claim \ref{29c}. Therefore,  $B_2(X_2)\cap T_0\ne \emptyset$. In what follows, we always suppose that $$v_0\in B_2(X_2)\cap T_0~~\text{and}~~u_0\in B_1(X_1).$$

If  $u_0\in B_1(X_1)\cap T_1$, then by Claim \ref{25c} we have  $(X_1\setminus \{u_0\}\cup \{v_0\}, X_2\setminus \{v_0\}\cup \{u_0\})\in  \mathcal{P}$ and $v_0\in T_1$, contradicting $v_0\in B_2(X_2)\cap T_0$.   Therefore,  $B_1(X_1)\subseteq T_0$.

If $B_1(X_1)\cap N_{X_1}(u_0)=\emptyset$, then there exists  $ x\in (X_1\setminus B_1(X_1))\cap N_{X_1}(u_0)$ such that   $ d_{X_1}(x)\ge a(x)+1 $, as $d_{X_1}(u_0)=a(u_0)\ge 1$. Since $ N_{X_1}(u_0)\subseteq X_1\setminus B_1(X_1)$,  $ X_1\setminus\{u_0\} $ is $ a $-nice by Claim \ref{29c}. Let $ X_1'=X_1\setminus \{u_0\} $ and $ X_2'=X_2\cup \{u_0\} $. By Claim \ref{26c}, there exists a $ 2 $-good subset $ B\subseteq X_2' $ such that $ u_0\in B $. Therefore, $ (X_1',B) $ is an $ (a,b) $-feasible pair. Therefore, $B_1(X_1)\cap N_{X_1}(u_0)\ne \emptyset$.

Gathering the previous conclusions, we arrive at the following setting: \begin{align}\label{liu1e}B_2(X_2)\cap T_0\ne \emptyset,~~u_0\in B_1(X_1)\subseteq T_0 ~~\text{and}  ~~B_1(X_1)\cap N_{X_1}(u_0)\ne \emptyset.\end{align}

We now consider separately the two cases arising from parts (i) and (ii) of the theorem.

\medskip\noindent{$\bullet$ \textit{Case 1: $H$ is  obtained   by connecting a single vertex to exactly two vertices of $ K_{4}-e$.}}

By \eqref{liu1e}, we have   $1\le  \left| N_{X_1}(u_0)\cap B_1(X_1)\right|\le 2$, as $B_1(X_1)\subseteq T_0$. We choose $y_0\in N_{X_1}(u_0)\cap B_1(X_1)$.  By the symmetry of $u_0$ and $y_0$, we have $1\le  \left| N_{X_1}(y_0)\cap B_1(X_1)\right|\le 2$.

To facilitate reading, we briefly outline the remainder of this part of the proof:
\begin{itemize}
	\item[(a)] we first establish  the forthcoming equalities~\eqref{eq:cl},
	\item[(b)] we then consider the case $X_1 = \{u_0, y_0\}$,
	\item[(c)] and finally the case $X_1 \neq \{u_0, y_0\}$,
\end{itemize}
showing in each instance that an appropriate $(a,b)$-feasible pair exists, thereby eliminating the case.

We proceed with (a), i.e., we claim that  \begin{equation}\label{eq:cl}|N_{X_1}(y_0)\cap B_1(X_1)|=|N_{X_1}(u_0)\cap B_1(X_1)|=1.\end{equation} Otherwise, assume that   $\max\{|N_{X_1}(y_0)\cap B_1(X_1)|, |N_{X_1}(u_0)\cap B_1(X_1)|\}= 2$. Without loss of generality, suppose that   $ N_{X_1}(y_0)\cap B_1(X_1)=\{u_0,w_0\} $. Since $B_1(X_1)\subseteq T_0$, we have   $B_2(X_2)=\{v_0\}$.

Assuming that $X_1=\{w_0,y_0,u_0\}$, we let $X'_1=\{w_0,y_0,v_0,u_0\}$ and $X'_2=V(G)\setminus X'_1$.  Every vertex $y\in X_2\setminus \{v_0\}$ is adjacent to at most one vertex of $\{w_0,y_0,v_0,u_0\}$, as $u_0\in T_0$. This implies $$d_{X'_2}(y)\ge a(y)+b(y)+h(y)-1\ge b(y)+1+h(y)-1=b(y)+h(y).$$ By Claim \ref{26c}, there {exists} a $ 1 $-good subset of $ X_1' $. Thus, {$(X'_1, X'_2)$ contains an $  (a, b) $-feasible pair}. Thus, $X_1\ne \{w_0,y_0,u_0\}$. Let $x\in X_1\setminus \{w_0,y_0,u_0\}$. If $x\in B_1(X_1)$, then $N_G(x)\cap \{w_0,y_0,u_0\}=\emptyset$, as $u_0\in T_0$.
If $x\not\in B_1(X_1)$, then $|N_G(x)\cap \{w_0,y_0,u_0\}|\le 1$ and $d_{X_1}(x)\ge a(x)+1$, as $u_0\in T_0$. Let $X''_1=X_1\setminus \{w_0,y_0,u_0\}$ and $X''_2=X_2\cup \{w_0,y_0,u_0\}$. Irrespective of whether $x\in B_1(X_1)$ or not, we have $d_{X''_1}(x)\ge a(x)$. By Claim \ref{26c}, there {exists} a $ 2 $-good subset of $ X_2'' $. Thus, {$(X''_1, X''_2)$ contains  an $(a, b) $-feasible pair}. This confirms our claim.

With \eqref{eq:cl}, we first consider the case (b):  $ X_1=\{u_0,y_0\} $. Note that $ \left| V(G)\right|\ge 5  $, which shows $ \left| X_2\right|\ge 3$. If $ \left| N_G(v)\cap \{v_0,y_0,u_0\}\right|\le1 $ for every $ v\in X_2\setminus\{v_0\} $, then we set $ X_1'=\{v_0,y_0,u_0\} $ and $ X_2'=X_2\setminus\{v_0\} $. In this setting, $ X_1' $ contains an $ a $-nice set by Claim \ref{26c}, which in light of Claim~\ref{21c} implies $$ d_{X_2'}(v)\ge d_G(v)-d_{X_1'}(v)\ge a(v)+b(v)+h(v)-1\ge  1+b(v)+h(v)-1= b(v)+h(v),$$  for  $ v\in X_2' $. Hence, {$ (X_1',X_2') $   contains  an $(a, b) $-feasible pair}.
 
  Hence, there {exists} a vertex $ v\in X_2\setminus\{v_0\} $ such that $ \left| N_G(v)\cap \{v_0,y_0,u_0\}\right|\ge2 $. Since $u_0\in T_0$ by \eqref{liu1e},   there is exactly one vertex, say $v_1 $, in $ X_2\setminus\{v_0\} $ such that $ \left| N_G(v_1)\cap \{v_0,y_0,u_0\}\right|\ge 2 $ and $ \left| N(v_2)\cap \{v_0,v_1,y_0,u_0\}\right|\le1 $ holds for every $v_2\in X_2\setminus\{v_0,v_1\} $. Since  $$ d_{X_2\setminus\{v_0,v_1\}}(w)=d_G(w)-d_{X_1\cup\{v_0,v_1\}}(w)\ge a(w)+b(w)+h(w)-1\ge 1+b(w)+h(w)-1= b(w)+h(w)$$ for  $ w\in X_2\setminus\{v_0,v_1\} $,  Claim~\ref{26c} ensures that {$ (\{ u_0,y_0,v_0,v_1\}, X_2\setminus\{v_0,v_1\}) $ contains an $ (a,b) $-feasible pair}.

 We now consider the case (c) in which $\{u_0, y_0\}$ is a proper subset of $X_1$, and set $x\in X_1\setminus\{u_0,y_0\}$. We claim that there exists $x_0\in X_1\setminus\{u_0,y_0\}$ such that $x_0\not\in B_1(X_1)$ and $ \left| N(x_0)\cap \{u_0,y_0\}\right|=2$. Otherwise, assume that $ \left| N(x)\cap \{u_0,y_0\}\right|\le 1$ holds for every $x\in X_1\setminus B_1(X_1)$. If $y\in B_1(X_1)$, then $ \left| N(y)\cap \{u_0,y_0\}\right|=0$, as $|N_{X_1}(y_0)\cap B_1(X_1)|=|N_{X_1}(u_0)\cap B_1(X_1)|=1$. Combining this with the previous inequality, we deduce that $X_1\setminus\{u_0,y_0\} $ and $X_2\cup\{u_0,y_0\}$ create an  $ (a,b) $-feasible pair, by Claim \ref{26c}. This confirms the claim. 
 
 Moreover, since  $u_0\in T_0$,  $x_0$ is the unique vertex of   $ X_1\setminus B_1(X_1) $ such that  $ \left| N(x_0)\cap \{u_0,y_0\}\right|=2$. The possibility $ d_{X_1}(x_0)\ge a(x_0)+2 $ is eliminated by the existence of an $(a, b)$-feasible pair $ (X_1\setminus\{u_0,y_0\}, X_2\cup\{u_0,y_0\}) $. Hence, $d_{X_1}(x_0)=a(x_0)+1 $.

If $ X_1\setminus\{u_0,y_0,x_0\}\neq\emptyset $, then  $ \left|N(v)\cap\{u_0,y_0,x_0\} \right|\le 1 $ for  $ v\in X_1\setminus\{u_0,y_0,x_0\}$, as $u_0\in T_0$. In addition $u_0\in T_0$ implies $N_G(x_0)\cap B_1(X_1)=\{u_0,y_0\}$.  Hence, $ X_1\setminus\{u_0,y_0,x_0\} $ is $ a $-nice, and Claim~\ref{26c} ensures that  {$(X_1\setminus\{u_0,y_0,x_0\}, X_2\cup \{u_0,y_0,x_0\})$ contains an $(a,b)$-feasible pair}. Thus,  $ X_1=\{u_0,y_0,x_0\} $.

 Since $u_0\in T_0$, we have $B_2(X_2)=\{v_0\}$ and $ \left|N(v)\cap\{v_0,y_0,x_0,u_0\} \right|\le 1 $ holds  for every $ v\in X_2\setminus\{v_0\} $. We set $X'_1=X_1\cup \{v_0\}$ and $X'_2=X_2\setminus \{v_0\}$. In this context $$d_{X'_2}(v)\ge a(v)+b(v)+h(v)-1\ge 1+b(v)+h(v)-1=b(v)+h(v)$$ holds for every $ v\in X'_2$, which, together with Claim~\ref{26c}, implies that  { $ (X'_1,X'_2) $ contains an $ (a,b) $-feasible pair}. This completes the last possibility for $X_1$, thereby finalizes the proof of Case 1.

\medskip\noindent  {$\bullet$ \textit{Case 2: $H_0\cong K_{2,3}$.}}

  If	$ \left|B_1(X_1)\right| =1$, then   $ B_1(X_1)=\{u_0\} $.    For every $ x\in X_1\setminus B_1(X_1) $, since $ d_{X_1}(x)\ge a(x)+1 $,    $ X_1\setminus\{u_0\} $ is $ a $-nice. Combining this with   Claim \ref{26c}, {$ (X_1\setminus\{u_0\}, X_2\cup \{u_0\}) $ contains an $ (a,b) $-feasible pair}. Thus,
$ \left|  B_1(X_1)\right| \ge 2 $.

  Since $u_0\in T_0$ by \eqref{liu1e}, we have   $|  B_2(X_2)|\le 2$ by Claim \ref{24c}. We first  claim that $ \left|  B_2(X_2)\right|=1 $. Otherwise, suppose that $ \left|  B_2(X_2)\right|=2 $. Since $u_0\in T_0$, we also have $ \left|  B_1(X_1)\right|=2.$
Since $d_{X_1}(u_0)=a(u_0)\ge 2-h(u_0)=2$, we have   $X_1\setminus B_1(X_1)\ne \emptyset$. Now, $x\in  X_1\setminus B_1(X_1)$ is adjacent to at most one vertex of $B_1(X_1)$, as   $u_0\in T_0$. This implies that $X_1\setminus B_1(X_1)$ is $a$-nice. By Claim \ref{26c},  $(X_1\setminus B_1(X_1), X_2\cup   B_1(X_1))$ {contains an $(a,b)$-feasible pair}, which confirms the claim.

Therefore, $ \left|  B_2(X_2)\right|=1$, say $ B_2(X_2)=\{v_0\}$.  By \eqref{liu1e}, we also have $v_0\in T_0$.

    We now claim the existence of a vertex $v\in T_0\cap X_2 $ such that $v_0v\in E(G) $ and $d_{X_2}(v)=b(v)$. Let $ X'_1=X_1\cup\{v_0\} $ and $ X'_2=X_2\setminus\{v_0\} $. By Claim \ref{26c},   there exists a $ 1 $-good subset $ A $ of  $ X'_1$ containing $v_0$. Since $G$ does not admit an $(a,b)$-feasible pair, there must {exists} a vertex $ v\in X'_2$ such that $ d_{X'_2}(v)\le b(v)-1 $. Hence, we have \begin{align*}
	b(v)+h(v)\le d_{X_2}(v)=	d_{X'_2}(v)+e(v_0,v)\le b(v)-1+1=b(v),
\end{align*}
as $v\not\in B_2(X_2)$. 
This forces that  $ e(v_0,v)=1 $, $ h(v)=0 $ and $ d_{X_2}(v)=b(v) $, as  desired.

 If
there exists a vertex $x_0$ of $B_1(X_1) $  such that $x_0v\notin E(G)$, then   $(X_1\setminus\{x_0\}\cup\{v_0\}, X_2\setminus\{v_0\}\cup\{x_0\}) \in\mathcal{P} $ by Claim \ref{25c}, as $B_1(X_1)\subseteq T_0$.  Since $ d_{X_2}(v)=b(v)$ and $x_0v\notin E(G)$, we have  $ x_0,v\in B_2(X_2\setminus\{v_0\}\cup\{x_0\})$ by Claim \ref{29c}.  The remainder follows the line of the previous part of the proof eliminating $|B_2(X_2)|\ge 2$. Thus, $v$ is adjacent to every vertex of $B_1(X_1)$. 

  Since $v_0\in T_0$, we have  $$ \left| X_1\right|\ge d_{X_1}(v_0)=a(v_0)+b(v_0)+h(v_0)-(b(v_0)-1)\ge 2-h(v_0)+h(v_0)+1=3.$$  
  Since $u_0\in T_0$ and $v$ is adjacent to every vertex of $B_1(X_1)$, we have $ \left| B_1(X_1)\right|=2$ and every $x\in X_1\setminus B_1(X_1)$ is adjacent to at most one vertex of $B_1(X_1)$. This implies that { $(X_1\setminus B_1(X_1), X_2\cup B_1(X_1))$ contains an $(a,b)$-feasible pair}. This was the final contradiction, completing Case 2 and the entire proof.\qed

\section*{Declaration of competing interest}

	The authors declare that there is no competing interest.

\section*{Data availability}

	No data was used for the research described in the article.
	
	\section*{Acknowledgement} 
	This
	research is supported by the Natural Science Foundation of Guangdong Province 	(No.~2024A1515011899), the National Natural Science Foundation of China (No.~12271182), the Science and Technology Development Fund, Macau SAR (No.~0080/2023/RIA2) and the Ministry of
	Science, Technological Development and Innovation of the Republic of Serbia (No. 451-03-136/2025-03/200104).

 \end{document}